\documentclass[11pt, oneside]{article}   	
\usepackage[pass,paperwidth=5.5in,paperheight=8.5in, top=0.35in, bottom=0.5in, left=0.35in, right=0.35in]{geometry}
\usepackage{graphicx}
\usepackage{amssymb}					
\usepackage{amsmath}					
\usepackage{mathrsfs}
\usepackage{tikz}
\usepackage{epstopdf}
\usepackage{amsthm}
\newtheorem*{theorem*}{Theorem}
\newtheorem*{lemma*}{Lemma}
\newtheorem*{example*}{Example}
\newtheorem{theorem}{Theorem}

\def\b0{{\bf 0}}
\def\b1{{\bf 1}}

\def\cD{{\cal D}}

\def\cX{{\cal X}}

\def\cO{{\cal O}}
\def\cQ{{\cal Q}}

\def\n{\noindent}

\begin{document}
{\centering
\begin{Large}
{\bf FACTORIZATION OF  PLATONIC POLYTOPES INTO CANONICAL SPHERES}
\end{Large}
\footnote{appeared in {\it Geombinatorics},  {\bf XXXI}(1)(2021) 5--9.}
\vspace{0.3in}

Richard H.~Hammack\\
Department of Mathematics and Applied Mathematics\\
Virginia Commonwealth University\\
Richmond, VA  23284-2014 USA\\
\texttt{rhammack@vcu.edu}

\vspace{0.3in}
Paul C.~Kainen\\
Department of Mathematics and Statistics\\
Georgetown University\\
Washington, D.C. 20057 USA\\
 \texttt{kainen@georgetown.edu}
\par}



\vspace{0.3in}

\begin{abstract}
Factorization into spheres is achieved for skeleta of the simplex, cube, and cross-polytope, both explicitly and using Keevash's proof of existence of designs.
\end{abstract}

\noindent
{\bf Key Phrases}: {Decomposition of skeleta, 
Steiner triples, Hanani quadruples}

\bigskip
Decomposing a polytope's $2$-skeleton (i.e., partitioning its $2$-cells) into closed manifolds was studied in \cite{bsw76}, \cite{spreer}, \cite{hk-bhms}, \cite{hk-monthly}, \cite{pk-tubes-and-cubes}, while in \cite{hk-ADAM} we found {\it factorizations} of hypercube $2$-skeleta into boundaries of (pairwise isomorphic) $3$-cubes. Here we obtain such sphere factorizations for $k$-skeleta of all (non-exceptional) Platonic polytopes both explicitly for low values of $k$ and, as a consequence of Keevash's result \cite{keevash}, also existentially with $k \geq 1$ arbitrary.  Note we only use the case of multiplicity $\lambda = 1$ of \cite{keevash}.

Let $\Delta_n$ denote the $n$-dimensional simplex, whose 1-skeleton is the graph $K_{n+1}$.
Let $\cO_n$ denote the $n$-dimensional cross-polytope, whose 1-skeleton is the graph $K_{2n} - F$, where $F$ is a 1-factor.  Let $\cQ_n$ denote the $n$-dimensional hypercube, whose 1-skeleton is the graph $Q_n$, the $n$-fold graph Cartesian product of $K_2$.  These are the three families of Platonic polytopes.

For any $n$-dimensional polytope $A$ and  nonnegative integer $k$, let $A^k$ denote the $k$-skeleton of $A$; see, e.g., \cite{coxeter}. If $L$ and $K$ are polytopal $\ell$-complexes, let $L | K$ mean that {\bf L factors K}; that is,  $K$ and $L$ are of equal dimension $\ell$ and $K$ is the union of subcomplexes $L_1, \ldots, L_r$, each isomorphic to $L$, such that every $\ell$-face of $K$ is contained in exactly one of the $L_i$.
An $\ell$-complex is {\bf even} if each  $(\ell{-}1)$-face is in a positive even number of $\ell$-faces. 


For $X \in \{\Delta, \cO, \cQ\}$
we study the factorization of $K = X_n^\ell$ by $L = X_{\ell + 1}^\ell$ for
$\ell \leq n{-}1$.
We call 
$L$ the {\bf canonical} $\ell$-sphere as it is the boundary of a type $X$ polytope of dimension $\ell {+} 1$.  For $\ell \geq 1$, such factorizations always exist for $\cO$ as the proof of the first theorem below shows explicitly.

Observe, however, that these canonical spheres might not be composed of the {\it minimum} number of $\ell$-faces.
Indeed, by Theorem \ref{th:O} below,  the 1-skeleton of $\cO_3$ is factored into three
4-cycles, rather than four 3-cycles.

Clearly, $X_n^\ell$ even is a necessary condition for a factorization into spheres.  The $\ell$-skeleton of an $n$-simplex (or an $n$-cube) is an even complex if and only if $n{-}\ell{+}1$ (the number of $\ell$-faces containing an $(\ell{-}1)$-face) is positive and even; for the cross-polytope $\cO_n$, 
the $\ell$-skeleton is even for all $1 \leq \ell < n$.
Evenness is automatic for $\cO_n^\ell$ as each $(\ell{-}1)$-face is a simplex and is part of two $\ell$-faces (also simplexes) one for each vertex in a copy of $\overline{K_2}$.

The following extends \cite{hk-bhms} for $\ell=2$; cf. Spreer \cite{spreer}, which partitions $\cO_n^2$ into face-disjoint surfaces that contain the 1-skeleton $\cO_n^1$; such surfaces are called {\bf 1-Hamiltonian}. Observe that our result does {\it not} depend on \cite{keevash}.

\begin{theorem}
If  $1 \leq \ell <n$, then $\cO_{\ell + 1}^\ell\,|\,\cO_n^\ell$. 
\label{th:O}
\end{theorem}
\begin{proof}
For each $k\geq1$, $\partial \cO_k$ is the iterated topological join of $k$  copies of the 2-point complex $\overline{K_2}$, which is a $(k{-}1)$-sphere. The $\ell$-skeleton of $\cO_n$ is the $\ell$-cell-disjoint union of the ${{n}\choose{\ell + 1}}$ copies of $\cO_{\ell + 1}^\ell$ formed by the iterated join of each $(\ell {+} 1)$-element subset of the set $n\,\overline{K_2}$.
\end{proof}
 


For the simplex, we use the theory of Steiner-type configurations $(v, k, \ell)$,  which
consist of a $v$-element set $X$
endowed with a family of $k$-element subsets of $X$ (called {\bf blocks}) such that each
$\ell$-subset of $X$ is in exactly one block.  The case $k = \ell{+}1$ will be what we apply to factor $\Delta_{v-1}^{\ell-1}$ by $\Delta_{\ell}^{\ell-1}$. 

The translation from design to combinatorial factorization works for special cases, such as the Steiner triple systems as we noted in \cite{hk-ADAM}. But by Keevash
 \cite{keevash} (see also Gowers \cite{gowers} and Kalai \cite{kalai}), this sort of  factorization must occur {\it almost always} for the $\ell$-skeleta of simplexes and, hence, of cubes, as we shall explain.
Write $a\,|\,b$ to mean $b$ is an integer multiple of $a$.
 
Define the {\bf divisibility set} (e.g.,  \cite{hanani}) 
to be 
the set of all feasible values 
\begin{equation}
\cD(k,\ell) := \Big\{v : {{k-h}\choose{\ell-h}}\,\Big|\,{{v-h}\choose{\ell-h}},\,0 \leq h \leq \ell{-}1 < k \leq v\Big \},
\end{equation}

By \cite{keevash}, 
{\it for $1 \leq \ell \leq k$,   
$\exists$ {\rm finite} set $\cX(k,\ell) \subset \cD(k,\ell)$ such that}
\begin{equation}
\exists (v,k,\ell)\mbox{{\rm -configuration}} \iff v \in \cD(k,\ell) \setminus \cX(k,\ell).
\end{equation} 
\n

\begin{figure}
\centering
\begin{tikzpicture}[style=thick,scale=1.4]
\foreach \x in {-2.25,0,2.25,4.5} \foreach \t in {0,1,2,3,4,5,6,7}
{
\draw [lightgray] (\x,0) +(\t*45+22.5:1)--+(\t*45+45+22.5:1)--+(\t*45+135+22.5:1);
\draw [lightgray] (\x,0) +(\t*45+22.5:1)--+(\t*45+135+22.5:1);
}
\foreach \x in {-2.25,0,2.25,4.5} \foreach \t in {0,1,2,3} \draw [lightgray] (\x,0) +(\t*45+22.5:1)--+(\t*45+180+22.5:1);

\foreach \x in {-2.25, 0,2.25,4.5} \foreach \t in {0,1,2,3,4,5,6,7}
\draw [lightgray,fill=white] (\x,0) +(\t*45+22.5:1) circle (0.08);

\foreach \x in {-2.25, 0,2.25,4.5} \foreach \t in {0,1,2,3,4,5,6,7}
\draw (\x,0) +(\t*45+22.5:1)  node {\tiny $\t$};

\draw [very thick] (-2.25,0) +(0+22.5:1)--+(45+22.5:1)--+(90+22.5:1)--+(135+22.5:1)--+(0+22.5:1)--+(90+22.5:1) +(45+22.5:1)--+(135+22.5:1);
\foreach \t in {0,45,90,135} 
\draw [fill=lightgray] (-2.25,0) +(\t+22.5:1) circle (0.08);
\foreach \t in {0,1,2,3} 
\draw (-2.25,0) +(\t*45+22.5:1) node {\tiny $\t$};

\draw [very thick] (-22.5:1)--(45+22.5:1)--(90+22.5:1)--(180+22.5:1)--(-22.5:1)--(90+22.5:1)
(90-22.5:1)--(180+22.5:1);
\foreach \t in {0,90,135,225} 
\draw [fill=lightgray] (\t-22.5:1) circle (0.08);
\foreach \t in {1,2,4,7} 
\draw  (\t*45+22.5:1) node {\tiny $\t$};

\draw [very thick] (2.25,0) +(90-22.5:1)--+(90+22.5:1)--+(-90-22.5:1)--+(-90+22.5:1)--+(90-22.5:1)
 +(90+22.5:1)--+(-90+22.5:1)  +(90-22.5:1)--+(-90-22.5:1);
\foreach \t in {90,135,270,315} 
\draw [fill=lightgray] (2.25,0) +(\t-22.5:1) circle (0.08);
\foreach \t in {1,2,5,6}  \draw (2.25,0) +(\t*45+22.5:1) node {\tiny $\t$};
 
\draw [very thick] (4.5,0) +(0+22.5:1)--+(90+22.5:1)--+(180+22.5:1)--+(270+22.5:1)--cycle;
\draw [very thick] (4.5,0) +(22.5:1)--+(180+22.5:1) +(90+22.5:1)--+(270+22.5:1);
\foreach \t in {0,90,180,270} 
\draw [fill=lightgray] (4.5,0) +(\t+22.5:1) circle (0.08);
\foreach \t in {0,2,4,6}  \draw (4.5,0) +(\t*45+22.5:1) node {\tiny $\t$};

\draw (-2.25,2) node {
\begin{minipage}{1in}
\footnotesize
Four rotations of this configuration
by angles of $k\frac{\pi}{2}$ for
$k=0,1,2,3$.
\end{minipage}};

\draw (-2.25,-1.75) node {
\begin{minipage}{1in}
\footnotesize
\centering
$0234$\\
2345\\
4567\\
6701
\end{minipage}};

\draw (0,2) node {
\begin{minipage}{1in}
\footnotesize
Four rotations of this configuration
by angles of $k\frac{\pi}{2}$ for
$k=0,1,2,3$.
\end{minipage}};

\draw (0,-1.75) node {
\begin{minipage}{1in}
\footnotesize
\centering
1247\\
3461\\
5603\\
7025
\end{minipage}};

\draw (2.25,2) node {
\begin{minipage}{1in}
\footnotesize
Four rotations of this configuration
by angles of $k\frac{\pi}{4}$ for
$k=0,1,2,3$.
\end{minipage}};

\draw (2.25,-1.75) node {
\begin{minipage}{1in}
\footnotesize
\centering
1256\\
2367\\
3470\\
4501
\end{minipage}};

\draw (4.5,2) node {
\begin{minipage}{1in}
\footnotesize
Two rotations of this configuration
by angles of $0$ and $\frac{\pi}{4}$ 
\end{minipage}};

\draw (4.5,-1.5) node {
\begin{minipage}{1in}
\footnotesize
\centering
0246\\
1357

\end{minipage}};
\end{tikzpicture}
\label{Fig:Delta}
\caption{Factoring $\Delta_7^2$ into 14 $2$-face-disjoint tetrahedral boundaries $\Delta_3^2$}
\end{figure}
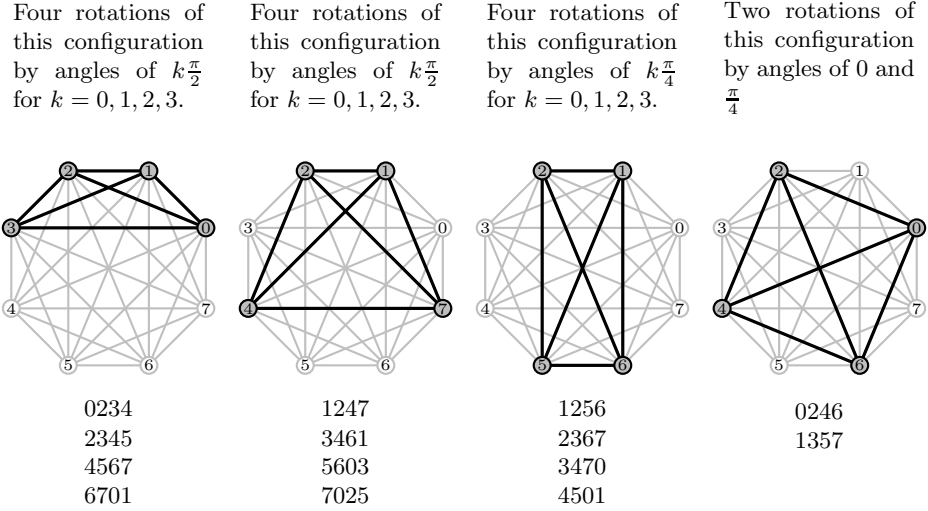

We interpret the combinatorics {\it geometrically} as applying to faces of the simplex.
Since the $(j{+}1)$-sets from the $n{+}1$ vertices of $\Delta_n$
correspond to the $j$-faces of $\Delta_n$,
we obtain a factorization into canonical spheres.
\begin{theorem}
For $1\leq \ell$ and, {\rm assuming Keevash's result \cite{keevash}},
\begin{equation}
\Delta_{\ell{+}1}^\ell\,\Big |\,\Delta_{n}^\ell
\,\iff \, n+1\, \in \, \cD({\ell {+} 2},\ell{+} 1) \setminus \cX({\ell {+} 2},\ell{+} 1).
\end{equation}
\label{th:gen-sx}
\end{theorem}
\vspace{-.6 cm}
Any $(v,3,2)$-configuration is a Steiner triple system (as observed in \cite{keevash}), and corresponds to a factorization of the $1$-skeleton of
$\Delta_{v-1}$ (i.e., $K_{v}$) into edge-disjoint 3-cycles,
the canonical  (and minimum) spheres in 1-dimension.  

For the case $(v,4,3)$, corresponding to the $(v{-}1)$-simplex, Hanani constructed a family of $4$-sets (quadruples) that cover each $3$-set exactly once, and he proved in \cite{hanani} that they exist for $v \geq 4$ if and only if $v \equiv 2$ or $4$ (mod-$6$), so again {\it independent of}$\,$ \cite{keevash}, we can characterize when the $2$-skeleton of an $n$-simplex is factored by the boundary of a tetrahedron:
\begin{equation}
\Delta_3^2 \,|\,\Delta_n^2\;\iff\; n\equiv 1 \,\mbox{or}\;3\;(\mbox{mod}\;6), \,n \geq 3.
\label{eq:hanani}
\end{equation}
See Fig. 1 for an illustration of the first non-trivial case of (\ref{eq:hanani}).

In order to factor cube-skeleta, we use an ``exponentiation'' method that goes back to Danzer \cite{danzer}, \cite{ks-DIMAC}.  If $\sigma = (i_1,\ldots,i_{k+1})$ is any $k$-simplex in the $n$-simplex with vertex set $[n{+}1]:=\{1,\ldots, n{+}1\}$, let
\begin{equation}
2^\sigma := A_1 \times \cdots \times A_{n+1},
\label{eq:exp-corr}
\end{equation}
where $A_i = [0,1]$ if $i \in \sigma$, $A_i = \{0,1\}$ if $i \notin \sigma$, and $\times$ is topological product.
Thus, $2^\sigma$ is a family of $2^{n-k}$ pairwise-vertex-disjoint $(k{+}1)$-dimensional cubes in the $(n{+}1)$-cube. Let $2^K$ be the union of the $2^\sigma$ over all simplexes $\sigma \in K$ for $K$ a simplicial complex.
Then we have (see K\"{u}hnel \& Schulz \cite{ks-DIMAC}) 

\begin{equation} \mbox{{\rm (a)}}\;\;2^{\Delta_{n-1}^{\ell-1}}\,=\,\cQ_n^\ell; \;\;\;\;\mbox{{\rm (b)}}\;\;2^{\partial K}\,=\,\partial \,(2^K);\;\;\mbox{{\it and}}\;\;
\label{eq:ks1}
\end{equation}

\begin{equation}
2^K\;\mbox{{\it is a manifold}}  \iff K \;\mbox{{\it is a combinatorial sphere.}}
\label{eq:ks2}
\end{equation}

By (\ref{eq:ks1}), (a) and (b),
a factorization of the $(\ell{-}1)$-skeleton of the $(v{-}1)$-simplex by the boundary of $\Delta_{\ell}$ maps under the exponential correspondence to a factorization of the $\ell$-skeleton of the $n$-cube into subcomplexes isomorphic to $\partial(\cQ_{\ell+1})$.
We used this in \cite{hk-ADAM} to factor the 2-skeleton of a cube by the boundary of a 3-cube exactly when Steiner triples exist.  
In \cite{hk-monthly} we used (\ref{eq:ks2}) to show that, for $n$ even, the Hamiltonian factorization of the graph $K_{n+1}=\Delta_n^1$ into spanning 1-spheres yields a factorization of the 2-skeleton of $\cQ_{n+1}$ into orientable minimum-genus surfaces which are 1-Hamiltonian.

Similarly,
Hanani's factorization (\ref{eq:hanani}) of the 2-skeleton of the simplex gives
\begin{equation}
\cQ_4^3\,|\,\cQ_n^3
\iff n \geq 4, \;n\equiv 2 \;\mbox{or}\;4\;(\mbox{mod}\;6).  
\label{eq:3-cube-factor} 
\end{equation}

Again, (\ref{eq:3-cube-factor}) doesn't depend on \cite{keevash}, while
by Theorem \ref{th:gen-sx}, the exponential correspondence (\ref{eq:exp-corr}), and using \cite{keevash}, we have



\begin{theorem}
For $1 \leq \ell$,  if $n \in \cD(\ell {+} 1,\ell) \setminus \cX(\ell {+} 1,\ell)$, then $\, \cQ_{\ell+1}^{\ell}\,|\,\cQ_{n}^{\ell}$.
\label{th:gen-cube} 
\end{theorem}
We conjecture that the condition is also necessary. 
Note that the result for $\ell = 1$ follows from the true though not stated case of Theorem \ref{th:gen-sx} for $\ell = 0$ since $\cD(2,1)$ is the set of positive even integers and $\cX(2,1) = \emptyset$.
 
By the results of \cite{keevash}, factorability into canonical spheres is ubiquitous for skeleta of Platonic polytopes.  A {\it geometric} explanation for this phenomenon 
(in special cases)
would imply the existence theorem (in those cases).  

\end{document}